\documentclass[11pt,twoside]{amsart}

\usepackage{amsmath}
\usepackage{amsthm}
\usepackage{amsfonts, amssymb}
\usepackage{mathrsfs}
\usepackage[all,line]{xy}
\usepackage{url}
\usepackage{graphics}
\usepackage{epstopdf}
\usepackage{comment}
\usepackage{xcolor}

\definecolor{thmcol}{RGB}{65, 102, 245}
\definecolor{citecol}{RGB}{65, 102, 245}
\definecolor{linkcol}{RGB}{65, 102, 245}
\definecolor{urlcol}{RGB}{65, 102, 245}

\usepackage{hyperref}

\usepackage{latexsym}
\usepackage[dvips]{graphicx}

\usepackage{tikz}
\usepackage{tkz-graph}
\usetikzlibrary{decorations.text,arrows,quotes}

\usepackage{tikz-cd}

\theoremstyle{plain}
\newtheorem{lemma}{Lemma}[section]

\newtheorem{thm}[lemma]{Theorem}
\newtheorem*{thm1}{Theorem \ref{main}}
\newtheorem*{thmp}{Theorem \ref{pesos}}

\newtheorem*{cor1}{Corollary \ref{genho}}
\newtheorem*{coro2}{Corollary \ref{ciclos}}
\newtheorem{corollary}[lemma]{Corollary}
\theoremstyle{remark}

\newtheorem{remark}[lemma]{Remark}

\theoremstyle{definition}
\newtheorem{definition}[lemma]{Definition}
\newtheorem{ej}[lemma]{Example}

\def\k{\mathcal{K}}

\def\Z{\mathbb{Z}}

\def\N{\mathbb{N}}

\def\pe{\mathcal{P}}
\def\kp{\k_{\mathcal{P}}}
\def\tk{\tilde{\k}}
\def\pg{P(\Gamma)}
\begin{document}

\title[Local indicability of groups with homology circle presentations]{Local indicability of groups with homology circle presentations}
\author[A.N. Barreto and E.G. Minian]{Agustín Nicolás Barreto and El\'\i as Gabriel Minian}

\address{Departamento  de Matem\'atica-IMAS (CONICET)\\
 FCEyN, Universidad de Buenos Aires\\ Buenos
Aires, Argentina.}

\email{abarreto@dm.uba.ar}
\email{gminian@dm.uba.ar}

\thanks{Researchers of CONICET. Partially supported by grants PICT 2019-2338 and UBACYT 20020190100099BA}

\begin{abstract}
We investigate conditions that guarantee local indicability of groups that admit presentations with the homology of a circle, generalizing a result of J. Howie for two-relator presentations. We apply our results to investigate local indicability of LOT groups and some classes of non-cycle-free Adian presentations, extending previous results in that direction by J. Howie and D. Wise.
\end{abstract}



\maketitle

\section{Introduction}
Recall that a group $G$ is locally indicable if each of its finitely generated nontrivial subgroups admits a homomorphism onto the infinite cyclic group $\Z$. This class of groups appeared in the work of Higman on the units of group rings \cite{hig} and has been intensively studied in the last four decades since the works of Brodskii \cite{br} and Howie \cite{h0,h1}. The theory of locally indicable  groups has connections with equations over groups \cite{br,h0,km}, with complexes of non-positive curvature and coherence \cite{Wil,wi,wi2,hs}, and with dynamics and left orderable groups (see for example \cite{cr,rr}). They are also related to the study of asphericity of $2$-complexes and Whitehead's conjecture, since any connected $2$-complex $X$ with $\pi_1(X)$ locally indicable and $H_2(X)=0$ is aspherical (see \cite{h1,h2}). In fact,  locally indicable groups are $\Z$-conservative \cite{ge0,hsc}.

It is known that torsion-free one-relator groups are locally indicable. Moreover one-relator products of locally indicable groups are locally indicable if the relator is not a proper power (see \cite{h1}). In \cite[Theorem 6.2]{h2} Howie proved local indicability of groups admitting presentations $\pe=\langle a,b,c \ | \ a^{-1}w_1^{-1}bw_1,\ b^{-1}w_2^{-1}cw_2 \rangle$ with three generators and two relators, where $w_1$ is a specific type of word called sloping word. These presentations  arise naturally when studying (weak) labelled oriented trees (LOTs), and the result is used to prove local indicability of LOT groups for LOTs of diameter $3$ and some families of LOTs of diameter $4$ (see \cite{h2}).

One of the main results of this paper is a generalization of Howie's result \cite[Theorem 6.2]{h2} to homology circle presentations of any length where the relators  satisfy some extra hypotheses (weaker than Howie's original result for two-relator presentations). It is clear that some extra condition is needed to derive local indicability, since one can construct homology circle presentations of non locally indicable groups just by adding a new generator to any balanced presentation of a perfect group. The condition that we require depends on the minima or maxima of the total exponents of the initial subwords of all but one of the relators (see Definitions \ref{defmulti} and \ref{defconcatenable}).
\begin{thm1}
	Let $\mathcal{P} = \langle a_1, \ldots, a_{k+2} \ | \ r_1, \ldots, r_k, s \rangle$ be a presentation (of deficiency $1$) of a group $G$ with $H_1(G)=\Z$  (for some $k\geq 0$), where all the relators are cyclically reduced and have total exponent $0$. If the multisets of minima $m(r_1), \ldots, m(r_k)$ are concatenable, then $G$ is locally indicable. 
\end{thm1}
The proof essentially follows Howie's original proof of  \cite[Theorem 6.2]{h2}, adapted to a more general context. One can recover Howie's result as a particular case. 

The condition on the relators having total exponent $0$ can be dropped by considering any surjective homomorphism $\varphi:G\to\Z$ and the sequences of minima or maxima with respect to $\varphi$.
\begin{thmp}
	Let $\mathcal{P} = \langle a_1, \ldots, a_{k+2} \ | \ r_1, \ldots, r_k, s \rangle$ be a presentation (of deficiency $1$) of a group $G$ with $H_1(G)=\Z$ (for some $k\geq 0$), where all the relators are cyclically reduced. If there is a surjective homomorphism  $\varphi : G \to \mathbb{Z}$ such that $m_{\varphi}(r_1), \ldots, m_{\varphi}(r_k)$ are concatenable, then $G$ is locally indicable.
\end{thmp}

For two-relator presentations of deficiency $1$ one can widely extend the result by using relative maxima or minima (see Definition \ref{relativemin}). 
 \begin{cor1}
  Let $\mathcal{P} = \langle a, b, c \ | \ r, s \rangle$ be a presentation of a group $G$ with $H_1(G)=\Z$ where the relators are cyclically reduced. Let $\varphi : G \to \mathbb{Z}$ be a surjective homomorphism with $\varphi(c) \neq 0$. If $r$ attains a unique relative minimum at $a$, then $G$ is locally indicable.	
 \end{cor1}

In Theorem \ref{weak} we extend our methods to investigate presentations with the homology of a wedge of circles.

Our main applications are related to LOT groups and Adian presentations. Labelled oriented trees  give rise to homology circle presentations that generalize Wirtinger presentations for knots. It is well-known that knot groups are locally indicable (see \cite{h1}). Local indicability of LOT groups is an open problem (that would imply asphericity of the associated presentations). In fact, it is not even known whether all LOT groups are torsion-free. In \cite{h2} Howie associated to any LOT $\Gamma$ two graphs. The right graph $I(\Gamma)$ is obtained from $\Gamma$ by interchanging the initial vertex and the label of every edge in $\Gamma$. Analogously, the left graph $T(\Gamma)$ is constructed by interchanging the final vertex with the label. 
A similar construction was investigated by Gersten in \cite{ge} in the more general context of Adian presentations $\pe=\langle A \ | \ u_i=v_i \ (i\in J)\rangle$ (where $u_i$ and $v_i$ are nontrivial positive words). Howie proved that, if $I(\Gamma)$ or $T(\Gamma)$ has no cycles, then the group $G(\Gamma)$ is locally indicable (and, in particular, $\Gamma$ is aspherical) \cite[Theorem 10.1]{h2}. For Adian presentations, Gersten showed that if both $T(\pe)$ and $I(\pe)$ have no cycles then the associated $2$-complex $\kp$ is diagrammatically reducible and, in particular, aspherical \cite[Proposition 4.12]{ge}. Moreover, when $\ell(u_i)=\ell(v_i)$ for every $i$ (as in the case of LOT presentations), if either  $T(\pe)$ or $I(\pe)$ has no cycles then $\kp$ is diagrammatically reducible (and hence aspherical) \cite[Proposition 4.15]{ge}. More recently, Wise proved that the $2$-complexes associated to cycle-free Adian presentations (i.e. those for which both graphs are forests) have non-positive sectional curvature, which implies that the corresponding groups are locally indicable \cite[Theorem 11.4]{wi}.
For LOT presentations we obtain the following extension of Howie's result.
\begin{coro2}
	Let $\Gamma$ be a LOT. If either $I(\Gamma)$ or $T(\Gamma)$ has at most one cycle, then $G(\Gamma)$ is locally indicable.
\end{coro2}

In Corollary \ref{adian} we extend this result to a wider class of Adian presentations, generalizing, for a certain subfamily, Wise's result  \cite[Theorem 11.4]{wi}.

Finally, Theorem \ref{ciclosetiq} describes a strategy to study local indicability of LOT groups when both graphs   $I(\Gamma)$ and $T(\Gamma)$ have more than one cycle. This result is a consequence of Theorem \ref{weak}.

\subsection*{Acknowledgements}
We would like to thank Jim Howie and Jonathan Barmak for useful suggestions and comments during the preparation of this article.

\section{Main results}

The results of this section are based on some ideas and techniques developed by Jim Howie in the proof of \cite[Theorem 6.2]{h2} and use well-known properties of locally indicable groups. Given a presentation $\pe$ of a group $G=G(\pe)$, we denote by $\kp$ its standard $2$-complex which has a single $0$-cell, one $1$-cell for each generator and one $2$-cell for each relator. Recall from \cite{h1} that an elementary reduction $(X,Y)$ is a pair of $2$-complexes such that $Y \subset X$ and $X-Y$ consists of exactly one $1$-cell $e^1$ and at most one $2$-cell $e^2$, where $e^2$ properly involves $e^1$, which means that its attaching map is not homotopic in $Y \cup e^1$ to a map into $Y$. A $2$-complex $K$ is reducible if for every finite subcomplex $X$ there exists an elementary reduction $(X,Y)$. Also, a group presentation $\pe$ is reducible if its associated  $2$-complex $\kp$ is. We will use the following known fact.

\begin{thm}[Howie \cite{h1}] \label{red}
	If $G$ has a reducible presentation in which no relator is a proper power, then G is locally indicable.
\end{thm}

Let $F$ be a free group on a set $A$. Recall that the total exponent of a word $S=a_1^{\varepsilon_1}\ldots a_n^{\varepsilon_n}$ (with $\varepsilon_i\in \{\pm 1\}$) is $exp(S)=\varepsilon_1+\ldots+\varepsilon_n$. The length of a word $S$ will be denoted by $\ell(S)$. An initial subword of a word  $S=a_1^{\varepsilon_1}\ldots a_n^{\varepsilon_n}$ is a word of the form  $a_1^{\varepsilon_1}\ldots a_i^{\varepsilon_i}$ for some $1\leq i\leq n$. Suppose $\pe$ is a finite presentation of a group $G$ such that the total exponents of all its relators are zero. Let $\varphi:G\to\Z$ be the surjective homomorphism which sends every generator of $\pe$ to $1\in \Z$. Let $\tilde \k$ be the infinite cyclic cover of $\kp$ determined by $\varphi$. Similarly as in \cite{h1}, we index its $0$-cells with the integers. Via $\varphi:G\to\Z$  every generator $a$ in $\pe$ determines one $1$-cell $a_j$ oriented from $j$ to $j+1$ for each $j \in \mathbb{Z}$. Given a relator $r$, we denote by $R$ the corresponding $2$-cell in $\kp$, and by $R_j$, the $2$-cell in $\tilde \k$ that covers $R$ and whose minimum traversed $0$-cell is $j$. 

For any relator $r$, the sequence of total exponents of its initial subwords represents the indices of the $0$-cells traversed by the attaching path of the corresponding $2$-cell in $\tilde \k$  starting at vertex $0$. Suppose that the minimum value attained in the sequence is $m\in\Z$. Note that this minimum can be attained more than once and that the rightmost letter of an initial subword where $m$ is attained must have exponent $-1$. This letter represents a $1$-cell joining the vertices $m$ and $m+1$ (and it is traversed backwards). The letters with exponent $1$ where the value $m+1$ is attained represent  $1$-cells oriented from $m$ to $m+1$ and they are traversed forwards. Note that if we invert the relator these roles are interchanged. This motivates the following definition.

\begin{definition}\label{defmulti}
        Let $r$ be a relator of a presentation $\pe$ and let $m$ be the minimum value of the sequence of total exponents of its initial subwords. We define the \textit{multiset of minima} $m(r)$ as the multiset that contains the letters of $r$ with exponent $-1$ that attain the value $m$ and the letters that appear with exponent $1$ in $r$ that attain the value $m+1$.
\end{definition}

    For example, if $r=a^{-1}b^{-1}cb^{-1}ab$, the total exponents are $-1,-2,-1,-2,-1,0$ and $m(r) = \{b,c,b,a\}$.

\begin{remark}
      We can formulate an analogous definition for maxima instead of minima. In fact, every result in this paper that holds for minima has also an analogous version for maxima.   
\end{remark}

\begin{definition}\label{defconcatenable}
Given multisets $A_1, \ldots, A_n$ we say that they are \textit{concatenable} if there is an ordering $A_{i_1}, \ldots, A_{i_n}$ such that, for every $1 \leq j \leq n$, there exists an element in $A_{i_j}$ with multiplicity $1$ that does not belong to the union of $A_{i_1}, \ldots, A_{i_{j-1}}$.
\end{definition}

\medskip

The following result is our first generalization of Howie's \cite[Theorem 6.2]{h2}.

\begin{thm}\label{main}
Let $\mathcal{P} = \langle a_1, \ldots, a_{k+2} \ | \ r_1, \ldots, r_k, s \rangle$ be a presentation (of deficiency $1$) of a group $G$ with $H_1(G)=\Z$  (for some $k\geq 0$), where all the relators are cyclically reduced and have total exponent $0$. If the multisets of minima $m(r_1), \ldots, m(r_k)$ are concatenable, then $G$ is locally indicable. 
\end{thm}

We recover Howie's result \cite[Theorem 6.2]{h2} as a particular case.

\begin{corollary}[Howie \cite{h2}]
	Let $\Gamma$ be a weakly labelled oriented tree with $3$ vertices, in which at least one label is a reduced sloping word. Then $G(\Gamma)$ is locally indicable. 
\end{corollary}

The corollary follows from the fact that the standard presentation of the group of a weakly labelled oriented tree with three vertices $\pe=\langle a,b,c \ | \ r,s \rangle$ is a homology circle and the sloping word condition assures that $m(r)$ attains a unique minimum or a unique maximum in its sequence. In particular, it is concatenable. Note that, in fact, for a two-relator presentation  $\pe=\langle a,b,c \ | \ r,s \rangle$  it is not required that $m(r)$ has a unique minimum in order to be concatenable. We just need that there exists a generator that appears only once in $m(r)$. In Corollary \ref{genho} we will show a generalization of this result for two-relator presentations.

 The proof of Theorem \ref{main} uses the following technical lemma. 

\begin{lemma}\label{techlemma}
	Let $F$ be a free group on a set $A$ and $S$ a cyclically reduced word in $F$ with $exp(S)=0$. Then $S$ cannot be factored as a product 
	$$S = w_0S_0w_0^{-1}w_rS_rw_r^{-1}w_{r-1}S_{r-1}w_{r-1}^{-1}\cdots w_1S_1w_1^{-1}$$
	with $w_0,w_1,\ldots, w_r,S_0,S_1,\ldots,S_r \in F\ (r\geq 0)$, satisfying the following conditions:
	\begin{itemize}
		\item $exp(S)=exp(S_i)=0$ for $0 \leq i \leq r$, 
         \item $exp(w_i)\geq 1$ for $1 \leq i \leq r$ and $exp(w_0)=1$,
		\item the total exponents of all the initial subwords of $S, S_i$ and $w_i$ are non-negative $(0 \leq i \leq r)$.
\end{itemize}
		\end{lemma}
\begin{proof}
	Suppose there exists such a factorization for some cyclically reduced word $S$. Over all possible factorizations of $S$ and $S^{-1}$ satisfying these conditions, we consider the cyclically reduced words $w_0,w_1,\ldots, w_r,S_0, S_1,\ldots,S_r$ that minimize the sum $$L = \sum_{i=0}^r (\ell(w_i)+\ell(S_i)).$$
	We can assume that the factorization that minimizes $L$ is one of $S$. Otherwise, we change $S$ by $S^{-1}$. We rewrite the equality as $$w_0^{-1}Sw_1S_1^{-1}w_1^{-1}w_{2}S_{2}^{-1}w_{2}^{-1}\ldots w_rS_r^{-1}w_r^{-1}w_0 = S_0.$$
		Note that $S_i^{-1}$ satisfies the same properties as $S_i$ for every $0 \leq i \leq r$.
	
	Let $c$ be the rightmost letter in $S$. Note that it must appear as $c^{-1}$ (with negative exponent) since the initial subwords of $S$ have non-negative total exponent and $exp(S)=0$. The key observation is that this appearance of $c^{-1}$ in $S$ must cancel with a $c$, since, in any other case, $w_0^{-1}S$ would be an initial subword of $S_0$ with total exponent $-1$, which contradicts the hypotheses. There are five types of possible cancellations for $c^{-1}$.

	\noindent
	\textbf{Case 1.} The rightmost letter $c^{-1}$ in $S$ cancels with a $c$ in some $S_i^{-1}$ with $i \geq 1$. In this case we can write $S_i^{-1} = UcV$ for some (possibly empty) words $U$ and $V$. The word located between the letter $c^{-1}$ in $S$ and the $c$ in $S_i^{-1}$ must be trivial, so  $w_1S_1^{-1}w_1^{-1}\ldots w_{i-1}S_{i-1}^{-1}w_{i-1}^{-1}w_iU$ is trivial, and in particular this word has total exponent $0$. Since the total exponent of $w_1S_1^{-1}w_1^{-1}\ldots w_{i-1}S_{i-1}^{-1}w_{i-1}^{-1}$ is $0$ and $exp(w_i) \geq 1$, we have that $exp(U) \leq -1$, but this is a contradiction since the subwords of $S_i^{-1}$ have non-negative total exponent.

	\noindent
	\textbf{Case 2.} The rightmost letter $c^{-1}$ in $S$ cancels with a $c$ in some $w_i = UcV$ with $i \geq 1$. In this case, the hypothesis implies that $w_1S_1^{-1}w_1^{-1}\ldots w_{i-1}S_{i-1}^{-1}w_{i-1}^{-1}U$ is trivial, so we can erase this subword from the original equality, replace $w_i^{-1}=V^{-1}c^{-1}U^{-1}$, and obtain $$w_0^{-1}ScVS_i^{-1}V^{-1}c^{-1}U^{-1}w_{i+1}\ldots w_rS_r^{-1}w_r^{-1}w_0 = S_0.$$
	Note that $exp(cV)=1$, since $exp(w_i)=exp(UcV)\geq 1$ and $exp(U)=0$. Also, the total exponents of the initial subwords of $cV$ are non-negative because this property holds for $w_i$ and $exp(U)=0$. If $U$ is empty, we obtain a strictly shorter writing of $S$  where $cV$ plays the role of $w_i$, which is a contradiction. If $U$ is not empty, we can replace in the last equality $U^{-1} = w_1S_1^{-1}w_1^{-1}\ldots w_{i-1}S_{i-1}^{-1}w_{i-1}^{-1}$ obtaining a new factorization for $S$ where $cV$ plays the role of $w_i$. Since $cV$ has a strictly shorter length than $w_i$, this would be a contradiction.

	\noindent
	\textbf{Case 3.} The rightmost letter $c^{-1}$ in $S$ cancels with a $c$ in some $w_i^{-1} = V^{-1}cU^{-1}$ with $i \geq 1$. This is similar to Case 2.

	\noindent
	\textbf{Case 4.} The rightmost letter $c^{-1}$ in $S$ cancels with a $c$ in $w_0^{-1} = V^{-1}cU^{-1}$. Note that this condition implies $cU^{-1}=S^{-1}$. After replacing, we get $$V^{-1}w_1S_1^{-1}w_1^{-1}\ldots w_rS_r^{-1}w_r^{-1}SV = S_0.$$ Since $exp(w_0)=exp(SV)=1$ and $exp(S)=0$ we get that $exp(V)=1$. Also, the total exponents of initial subwords of $V$ are non-negative because $w_0$ satisfies that condition and $exp(S)=0$. Taking inverses at both sides of the equality we get
	$$V^{-1}S^{-1}w_rS_rw_r^{-1}\ldots w_1S_1w_1^{-1}V = S_0^{-1},$$ 
	and since $S$ and the $S_i$ have total exponent $0$, this provides a factorization of $S^{-1}$ satisfying the conditions, with a strictly smaller value of $L$, which is again a contradiction.

	\noindent
	\textbf{Case 5.} The rightmost letter $c^{-1}$ in $S$ cancels with a $c$ in $w_0 = UcV$. This is similar to Case 4.
\end{proof}

\begin{proof}[Proof of Theorem \ref{main}]
We follow the proof of \cite[Theorem 6.2]{h2}. If $k=0$, $G$ is a torsion-free one-relator group and the result follows from \cite{br} (see also \cite[Corollary 4.3]{h1}). Now suppose that $k\geq 1$. Let $R_1, \ldots, R_k, S$ be the $2$-cells of $\k_{\mathcal{P}}$ corresponding to $r_1, \ldots, r_k$ and $s$ respectively. Let  $\varphi: G \to \mathbb{Z}$ be the (surjective) homomorphism that sends every generator to $1$ and let $\tilde\k$ be the corresponding infinite cyclic cover of $\kp$. It suffices to prove that $\pi_1(\tk)$ is locally indicable. 

 We index the $0$-cells of $\tilde\k$ with the integers and denote $a_{i,j}$ the $1$-cell that covers $a_i$ and joins the $0$-cells $j$ and $j+1$. It is oriented from $j$ to $j+1$. We denote by $R_{i,j}$ and $S_j$ the $2$-cells that cover $R_i$ and $S$ respectively for which the minimum $0$-cell is $j$. Since $\pi_1(\tk)$ is the direct limit of the fundamental groups of its finite connected subcomplexes, it suffices to prove that for every finite connected subcomplex $K\subset \tk$, there is connected subcomplex $L\subset \tk$ with $K \subset L$ and $\pi_1(L)$ locally indicable. The strategy is to include $K$ in a finite subcomplex $L$ that is homotopy equivalent to a $2$-complex $L'$, which in turn collapses to a reducible $2$-complex and then use Theorem \ref{red}.

We define first the iterated rewrites of the cells $S_l$. Since $m(r_1), \ldots, m(r_k)$ are concatenable, we can assume without loss of generality that $a_i \in m(r_i)$ with multiplicity $1$ and $a_i \notin \cup_{j=1}^{i-1} m(r_j)$ ($1 \leq i \leq k$). Equivalently, for every $1 \leq i \leq k$, $a_{i,0}$ appears exactly once in the attaching path of $R_{i,0}$ and does not appear in the attaching path of $R_{j,0}$ for any $j < i$. Also $a_{i,0}$ passes through the $0$-cell $0$ once. We can then write the attaching path of $R_{i,0}$ as $a_{i,0}p_{i,0}^{-1}$ where $p_{i,0}$ is a path which does not involve $a_{i,0}$. Let $S_0^{(1)}$ be the path obtained after replacing, for each $1\leq i \leq k$ in decreasing order, every ocurrence of $a_{i,0}$ in the attaching path of $S_0$ by $p_{i,0}$ and cyclically reducing. Note that $S_0^{(1)}$ does not pass through $a_{i,0}$ for any $1 \leq i \leq k$. Analogously, $a_{i,1}$ appears just once in the attaching path of $R_{i,1}$ and so it can be written as $a_{i,1}p_{i,1}^{-1}$ where $p_{i,1}$ is a path that does not involve $a_{i,1}$. Let $S_0^{(2)}$ be the path obtained after replacing in $S_0^{(1)}$ every occurrence of $a_{i,1}$ by $p_{i,1}$, for each $1 \leq i \leq k$ in decreasing order, and cyclically reducing. Denote $S_1^{(1)}$ the path obtained by doing the same in $S_1$. By repeating this procedure, we define the iterated rewrite $S_l^{(j)}$ for each $l \geq 0$ and $j \geq 1$. 

Let $\alpha_1$ be the minimum $0$-cell traversed by $S_0^{(1)}$. Note that $\alpha_1\geq 0$. In general, for every $j$, we denote by  $\alpha_j$ the minimum $0$-cell of $S_0^{(j)}$. Note that $\alpha_j\geq \alpha_{j-1}$ and that the minimum  $0$-cell of $S_l^{(j)}$ is $\alpha_j+l$. We will show first that for $j\in\N$ big enough, the sequence  $\{\alpha_j\}_{j\geq 1}$ stabilizes. Since the sequence is increasing, we only need to prove that it is bounded above. We show that $\alpha_n<\rho+\sigma+1$ for every $n$, where $\rho = \sum_{i=1}^k \rho_i$ and  $\rho_1, \ldots, \rho_k, \sigma$ are the maximum $0$-cells traversed by the attaching paths of $R_{1,0}, \ldots,$ $R_{k,0}, S_0$ respectively.

Suppose that there exists $j\in\N$ such that $\alpha_j \geq \rho+\sigma+1$. Take $n\in\N$ big enough such that  $\alpha_{n-\rho-\sigma} \geq \rho+\sigma+1$. Let $Y$ be the full subcomplex of $\tk$ containing all $0$-cells $j$ for  $0 \leq j \leq n+\rho+\sigma$. This complex contains the $1$-cells $a_{i,j}$ for each $1\leq i \leq k+2$ and $0 \leq j \leq n+\rho+\sigma-1$, the $2$-cells $R_{i,j}$ for $1 \leq i \leq k$ and $0\leq j \leq n+\sigma+\rho-\rho_i$ and the $2$-cells $S_j$ with $0 \leq j \leq n+\rho$. Note that the number of $0$-cells in $Y$ is $n+\rho+\sigma+1$, the number of $1$-cells is $(k+2)(n+\rho+\sigma)$, for each $1\leq i \leq k$ it has $n+\sigma+\rho-\rho_i+1$ $2$-cells corresponding to the $R_{i,j}$ with $0 \leq j \leq n+\sigma+\rho-\rho_i$, and $n+\rho+1$ $2$-cells corresponding to the $S_j$ with $0 \leq j \leq n+\rho$. Therefore its Euler characteristic is $\chi(Y)= -\rho - \sigma + k+2$.

We construct a $2$-complex $Y'$ homotopy equivalent to $Y$ by replacing each $S_j$ by a $2$-cell attached via the path $S_j^{(n-j)}$ ($0\leq j \leq \rho+\sigma$). Note that this $2$-complex collapses to a subcomplex $Y''$ by removing $R_{i,j}$ and $a_{i,j}$ for every $1 \leq i \leq k$ and $0 \leq j \leq \rho + \sigma$. Let $Z$ be the $1$-subcomplex of $Y''$ that contains the $0$-cells from $0$ to $\rho + \sigma+1$ and the $1$-cells $a_{i,j}$  with $i=k+1,k+2$ and $0 \leq j \leq \rho+\sigma$. Since, for every $j$, the minimum $0$-cell of $S_j^{(n-j)}$ is $j+\alpha_{n-j} \geq j+\alpha_{n-\rho-\sigma}\geq j+\rho+\sigma+1$, the identity $Z \mapsto Z$ can be continuously extended to a retraction $Y'' \mapsto Z$ which sends all of $Y''-Z$ to the $0$-cell $\rho+\sigma+1$. It follows that 
$\beta_1(Y) = \beta_1(Y'') \geq \beta_1(Z)= \rho+\sigma +1$. Here $\beta_1$ denotes the first Betti number. Note that $\beta_0(Y)=1$ since $Y$ is connected. Note also that  $\kp$ is a homology circle because $H_1(\kp)=H_1(G)=\Z$ and the deficiency of the presentation is $1$. This implies that $\beta_2(Y) = 0$ since $\tk$ is an infinite cyclic cover (see \cite[Proposition 1]{ad1}). It follows that $\chi(Y) \leq -\rho-\sigma$, which is a contradiction. Therefore, the sequence $\{ \alpha_j\}_{j\geq 1}$ is bounded above by $\rho+\sigma+1$.

Now let $K\subset \tk$ be a finite connected subcomplex. We can assume that the minimum $0$-cell in $K$ is $0$. Let $q$ be the maximum $0$-cell in $K$ and take $n$ large enough such that $\alpha_i = \alpha$ for every $i\geq n$. Take $L$ to be the subcomplex of $\tk$ containing every $0$-cell between $0$ and $q+n+\rho-\sigma$, every $1$-cell $a_{i,j}$ with $1 \leq i \leq k+2$ and $0 \leq j \leq q+n+\rho-\sigma-1$, the $2$-cells $R_{i,j}$ with $1 \leq i \leq k$ and $0 \leq j \leq q+n+\rho-\rho_i-\sigma$ and the $2$-cells $S_j$ with $0 \leq j \leq q-\sigma$. Note that $K\subset L$. 

In order to prove that $\pi_1(L)$ is locally indicable, we replace each $S_j$ by a $2$-cell attached via $S_j^{(q+n-\sigma-j)}$ and obtain a new $2$-complex $L'$ homotopy equivalent to $L$ which collapses to a $2$-subcomplex $L''$ by removing every $2-$cell $R_{i,j}$ and every $1$-cell $a_{i,j}$ for each $1 \leq i \leq k$ and $0 \leq j \leq q+n+\rho-\rho_i-\sigma$. Note that the $2$-complex $L''$ is reducible: for $0 \leq j \leq q-\sigma$, the minimum $0$-cell of $S_j^{(q+n-\sigma-j)}$ is $j+\alpha$. The $2$-cell  $S_j^{(q+n-\sigma-j)}$ involves any of the $1$-cells in its attaching path incident to the $0$-cell $j+\alpha$ and these $1$-cells are not in the attaching paths of the remaining $2$-cells  $S_l^{(q+n-\sigma-l)}$ (for $l>j$). Some of these $1$-cells are properly involved by Lemma \ref{techlemma}. It follows that $\pi_1(L) = \pi_1(L'')$ is locally indicable by Theorem \ref{red} (see also \cite[Theorem 4.2]{h1}). Note that the no-proper-power condition follows from \cite[Proposition 1]{ad1} and the fact that $\tk$ is an infinite cyclic cover of a homology circle $2$-complex (see also \cite[p. 290]{h2}).
\end{proof}


\begin{ej}
    Consider $$\mathcal{P} = \langle a, b, c, d \ | \ a^{-1}c^{-1}b^{-1}a^{-1}badcb^{-1}d \ , \  c^{-1}a^{-1}b^{-1}d^{-1}badcb^{-1}d \ , \  a^{-1}b^{-1}c^{-1}acdb^{-1}c^{-1}ad \rangle$$
    Name the relators $r_1, r_2$ and $r_3$ respectively. Note that $m(r_1) = \{a,b\}$ and $m(r_2) = \{d,b\}$. Theorem \ref{main} (with $a_1=a$ and $a_2=d$) proves that $G(\pe)$ is locally indicable.
\end{ej}

Theorem \ref{main} can be generalized to presentations $\pe$ where the total exponents of the relators are not necessarily $0$. Let $\pe$ be a presentation and $\varphi: G(\pe) \to \mathbb{Z}$ a surjective homomorphism.  By changing, if necessary, a generator by its inverse, we can assume that $\varphi(a)$ is non-negative for every generator $a$ of $\pe$. Given a word $w$ on the generators, the weight of $w$ (with respect to $\varphi$) is $\varphi(w)$ (where $w$ is viewed as an element of $G(\pe)$). Similarly as before, we can consider the sequence of weights (with respect to $\varphi$) of initial subwords of a relator $r$ and define the multiset $m_{\varphi}(r)$ of minima of $r$. If the minimum of the sequence is $m$, $m_{\varphi}(r)$ contains a copy of the letter $a$ for every initial subword $wa^{-1}$ with weight $m$ and every initial subword $wa$ with weight $m+\varphi(a)$. In particular, if $\varphi$ sends every generator to $1$, $m_{\varphi}(r)=m(r)$. For example, if $\mathcal{P}= \langle a,b,c \ | \ abc^{-1}b^2, ab^{-3}a \rangle$ we can define $\varphi:G(\pe)\to\Z$ by $\varphi(a) = 3$, $\varphi(b) = 2$ and $\varphi(c)=9$. In this case, the sequence for the relator $r=abc^{-1}b^2$ is $3, 5, -4, -2, 0$.

\begin{thm}\label{pesos}
    Let $\mathcal{P} = \langle a_1, \ldots, a_{k+2} \ | \ r_1, \ldots, r_k, s \rangle$ be a presentation (of deficiency $1$) of a group $G$ with $H_1(G)=\Z$ (for some $k\geq 0$), where all the relators are cyclically reduced. If there is a surjective homomorphism  $\varphi : G \to \mathbb{Z}$ such that $m_{\varphi}(r_1), \ldots, m_{\varphi}(r_k)$ are concatenable, then $G$ is locally indicable.
\end{thm}

\begin{proof}
    We adapt the proof of Theorem \ref{main} to the more general setting. Let $\tk$ be the infinite cyclic cover of $\kp$ determined by $\varphi:G\to\Z$. The $1$-cells $a_{i,j}$ are oriented from $j$ to $j+\varphi(a_i)$ (recall that we assume that the weights $\varphi(a_i)$ are non-negative). As before, $\rho_i$ and $\sigma$ are the maximum $0$-cells of $R_{i,0}$ and $S_0$ respectively, where $R_{i,j}$ and $S_j$ are the $2$-cells of $\tk$ covering $R_i$ and $S$ with minimum $0$-cell $j$. 
    
    We will assume again that $a_i \in m_{\varphi}(r_i)$ with multiplicity $1$ and $a_i \notin \cup_{j=1}^{i-1} m_{\varphi}(r_j)$. This allows us to write the attaching path of $R_{i,j}$ as $a_{i,j}p_{i,j}^{-1}$ where $p_{i,j}$ is a path that does not involve $a_{i,j}$. Let $S_j^{(t)}$ be the iterated rewrite defined as in the proof of Theorem \ref{main} and again let $\alpha_j$ be the minimum $0$-cell of $S_0^{(j)}$.
    
    In order to prove that the sequence  $\{\alpha_j\}_{j\geq 1}$ stabilizes one has to consider a slightly different complex $Y$. We show that $\alpha_j< \rho+\sigma+\varphi(a_s)$ where $\varphi(a_s) = \max\{\varphi(a_i)\}$. Suppose $\alpha_j \geq \rho+\sigma+\varphi(a_s)$ for some $j$.  Take $n\in\N$ big enough such that  $\alpha_{n-\rho-\sigma} \geq \rho+\sigma+\varphi(a_s)$. Let  $\tilde Y$ be the full subcomplex of $\tk$ containing the $0$-cells between $0$ and $n+\rho+\sigma$. It contains the $1$-cells $a_{i,j}$ for $1\leq i \leq k+2$, $0 \leq j \leq n+\rho+\sigma - \varphi(a_i)$, the $2$-cells $R_{i,j}$ for $1 \leq i \leq k$ and $0 \leq j \leq n+\rho+\sigma-\rho_i$ and the $2$-cells $S_j$ for $0\leq j \leq n+\rho$. The complex $Y$ is obtained from $\tilde Y$ by attaching an extra $1$-cell between $j$ and $j+1$ for every $0 \leq j \leq \rho+\sigma+\varphi(a_s)-1$.

Now we construct a homotopy equivalent $2$-complex $Y'$ by replacing each $S_j$ by a $2$-cell attached via the path $S_j^{(n-j)}$ for each $0\leq j \leq \rho+\sigma$. Note that $Y'$ collapses to a subcomplex $Y''$ after removing $a_{i,j}$ and $R_{i,j}$ for every $1 \leq i \leq k$ and $0 \leq j \leq \rho + \sigma + \varphi(a_s) - \varphi(a_i)$. 
    The attaching of the extra $1$-cells guarantees that the full subcomplex $Z$ of $Y''$ that contains the $0$-cells between $0$ and $\rho+\sigma+\varphi(a_{s})$ is connected, and so is $Y$. Note that $Z$ is a $1$-subcomplex since  $\alpha_{n-\rho-\sigma} \geq \rho+\sigma+\varphi(a_s)$. By a similar argument to that of Theorem \ref{main}, $\beta_1(Y) \geq \beta_1(Z)= 2\rho+2\sigma + 2\varphi(a_s) - \varphi(a_{k+1}) - \varphi(a_{k+2}) + 2$. As before, this implies that $\chi(Y) \leq -1-2\rho-2\sigma -2\varphi(a_s) + \varphi(a_{k+1}) + \varphi(a_{k+2})$. On the other side, a direct computation shows that $$\chi(Y) = - 2\rho - 2\sigma - \varphi(a_s) + \sum_{i=1}^{k+2} \varphi(a_i)$$ which implies $\sum_{i=1}^{k} \varphi(a_i) + \varphi(a_s) \leq -1$, a contradiction.

    Now the proof proceeds similarly as in Theorem \ref{main}.
\end{proof}

\begin{ej}
   Let $\mathcal{P}= \langle a, b, c \ | \ c^{-1}b^{-1}c^{-1}abca \ , \ b^{-1}c^{-1}aca^{-1}bab^{-1} \rangle$. Using the homomorphism $\varphi:G(\pe)\to\Z$ defined by  $\varphi(a)=\varphi(b)=1$ and $\varphi(c)=2$, the multiset of minima of the first relator is $\{c,a\}$.  By Theorem \ref{pesos}, $G_{\mathcal{P}}$ is locally indicable.
\end{ej}

\subsection*{A generalization for two-relator presentations}
For presentations $\mathcal{P} = \langle a, b, c \ | \ r, s \rangle$, Theorem \ref{pesos} requires the existence of a surjective homomorphism $\varphi:G\to\Z$ such that $m_\varphi(r)$ (or $m_\varphi(s)$) is concatenable, which essentially means that there exists a generator that attains the minimum of the whole word (and only once). We show now that, by appling some Andrews-Curtis moves (or extended Nielsen transformations) to $\pe$,  this hypothesis can be relaxed: we only require the existence of a generator with a unique relative minimum.

\begin{definition}\label{relativemin}
    Let $\mathcal{P}$ be a presentation and  $\varphi:G(\pe)\to\Z$ a surjective homomorphism. By changing, if necessary, a generator by its inverse, we assume that $\varphi(a)$ is non-negative for every generator $a$ of $\pe$. We say that a relator $r$ attains a \textit{unique relative minimum at} $a$ if it has an initial subword of the form $w=\tilde wa^{-1}$ or $w=\tilde w a$, with $\varphi(w)=m_a$, that satisfies the following condition. If  $w=\tilde wa^{-1}$, then $\varphi(\hat wa^{-1})> m_a$ and  $\varphi(\hat wa)> m_a+\varphi(a)$ for every other initial subwords $\hat wa^{-1}$ and $\hat wa$.  If  $w=\tilde wa$, then $\varphi(\hat wa)> m_a$ and  $\varphi(\hat wa^{-1})> m_a-\varphi(a)$ for every other initial subwords $\hat wa^{-1}$ and $\hat wa$. 
\end{definition}

For example, if $r=a^{-1}bac^{-1}bbc^{-1}b$ and $\varphi(a)=1,\varphi(b)=2,\varphi(c)=4$, the sequence of weights is $-1,1,2,-2,0,2,-2,0$. Note that $m_\varphi(r)=\{c,c,b,b\}$ is not concatenable but there is a unique relative minimum at $a$ (with value $-1$).

\begin{corollary}\label{genho}
    Let $\mathcal{P} = \langle a, b, c \ | \ r, s \rangle$ be a presentation of a group $G$ with $H_1(G)=\Z$ where the relators are cyclically reduced. Let $\varphi : G \to \mathbb{Z}$ be a surjective homomorphism with $\varphi(c) \neq 0$. If $r$ attains a unique relative minimum at $a$, then $G$ is locally indicable.
\end{corollary}

\begin{proof}
    Let $m \in \mathbb{N}$ be arbitrary. Suppose that $r$ attains the relative minimum in an initial subword of the form $\tilde wa^{-1}$. We apply the following Nielsen transformations to $\pe$: we add a new generator $d$ together with a new relator $ac^{-m}d^{-1}c^m$, then we replace all the occurrences of the generator $a$ in $r$ and $s$  by  $c^{-m}dc^m$ (and cyclically reduce, if necessary) and obtain new relators  $\Tilde{r}$ and $\Tilde{s}$. Finally we remove the relator  $ac^{-m}d^{-1}c^m$ together with the generator $a$. We end up with an equivalent presentation $\pe' = \langle b,c,d \ | \ \Tilde{r}, \Tilde{s} \rangle$. Since the presentations are equivalent, we have $H_*(\pe')=H_*(\pe)$ and $G(\pe')=G(\pe)$. Note that $\varphi(d)=\varphi(a)$. We will show  that, for $m\in\N$ big enough, $\pe'$ satisfies the hypotheses of Theorem \ref{pesos}.

   For every subword $wa^{-1}$ with weight $k$ in $r$ there are $2m+1$ subwords 
   $$wc^{-1}, \ldots, wc^{-m}, wc^{-m}d^{-1}, wc^{-m}d^{-1}c, \ldots, wc^{-m}d^{-1}c^m$$
   in $\Tilde{r}$ with weights $k+\varphi(a)-\varphi(c), \ldots, k+\varphi(a)-m\varphi(c), k-m\varphi(c), k-(m-1)\varphi(c), \ldots, k$. Similarly, for every subword $wa$ with weight $k$ in $r$, we have $2m+1$ subwords $$wc^{-1}, \ldots, wc^{-m}, wc^{-m}d, wc^{-m}dc, \ldots, wc^{-m}dc^m$$ with weights $k-\varphi(a)-\varphi(c), \ldots, k-\varphi(a) - m\varphi(c), k-m\varphi(c), k-(m-1)\varphi(c), \ldots, k$. Since there is no initial subword of the form $\Tilde{w}a$ with weight $m_a+\varphi(a)$ or less, we can take $m$ big enough such that the minimum in the sequence of weights of $\Tilde{r}$ is unique and is attained at the $d$ that appears after replacing the rightmost $a^{-1}$ of the unique subword $\Tilde{w}a$ of weight $m_a$. The corresponding subword has weight $m_a - m \varphi(c)$.

If $r$ attains the relative minimum in an initial subword of the form $\tilde wa$, replace $r$ by $r^{-1}$.
\end{proof}

\begin{ej}
    Consider the presentation $$\mathcal{P} = \langle a, b, c \ | \ a^{-1}b^{-1}abaab^{-1}cb^{-1}a^{-1}a^{-1}a^{-1}baac^{-1} , c^{-1}b^{-1}cbccb^{-1}cb^{-1}c^{-1}c^{-1}c^{-1}baca^{-1} \rangle$$
    If we take $\varphi:G(\pe)\to\Z$ defined by $\varphi(a)=\varphi(b)=\varphi(c)=1$, the first relator attains a unique relative minimum at $c$. By Corollary \ref{genho}, $G(\pe)$ is locally indicable. One can check that the previous methods do not work for this presentation.
\end{ej}

\subsection*{Weakly concatenable relators and I-test}
Theorems \ref{main} and \ref{pesos} apply to presentations with the homology of $S^1$. Note that both results impose conditions on all the relators of the presentation but one. This fact will be used in the next section to extend a well-known criterion of Howie on LOT groups (see Corollary \ref{ciclos}). We now investigate local indicability of groups with presentations with the homology of a wedge of circles, i.e. groups $G$ with nontrivial free abelian $H_1(G)$ that admit presentations 
 $\mathcal{P} = \langle a_1, \ldots, a_{n} \ | \ r_1, \ldots, r_k \rangle$  where $n-k= rank(H_1(G))$. The next result imposes weaker conditions to all the relators of the presentation. 
 
 It will be convenient to use the notion of \textit{$I$-values}. This notion is inspired in the $I$-test developed in \cite{bm1}.  We will work with sequences of minima. As before, one can formulate and prove analogous results for maxima. 

\begin{definition}
    Let $\mathcal{P}$ be a presentation of a group $G$ and $\varphi: G \to \mathbb{Z}$ a surjective homomorphism. By changing, if necessary, a generator by its inverse, we can assume that $\varphi(a)$ is non-negative for every generator $a$ of $\pe$. Given a word $r$ on the generators,  the sequence of $I$-values of $r$ with respect to $\varphi$ is obtained by assigning to each initial subword $w$ with the form $w=\Tilde{w}x$ (where $x$ is a generator) the value $\varphi(\Tilde{w})$ and to each initial subword of the form $w=\Tilde{w}x^{-1}$ the value $\varphi(\Tilde{w}x^{-1})$.
\end{definition} 

This sequence is useful when considering the covering of $\kp$ corresponding to $\varphi$. The $I$-values of a relator $r$ are precisely the initial vertices of the $1$-cells traversed by the attaching path of $R_{m}$ (where $m$ is the minimum of its sequence of weights).

\begin{remark}
    Note that the unique relative minimum condition of Definition \ref{relativemin} is equivalent to the relator having an initial subword $wa^{-1}$ or $wa$ with minimum $I$-value $m_a$ over all initial subwords of the form $\Tilde{w}a$ or $\Tilde{w}a^{-1}$.
\end{remark}

\begin{remark}
    Given a relator $r$, the multiset of letters in which the minima of the sequence of $I$-values of $r$ are attained is exactly $m_{\varphi}(r)$.
\end{remark}

\begin{definition} A family of relators $\{r_1,\ldots,r_k\}$ is  \textit{weakly concatenable} (with respect to $\varphi:G\to\Z$), if 
    there is an ordering $m_{\varphi}(r_{i_1}), \ldots, m_{\varphi}(r_{i_k})$ of their multisets of minima such that, for every $1 \leq j \leq k$, there exists an element $x \in m_{\varphi}(r_{i_j})$ such that $x \notin \cup_{s=1}^{j-1} m_{\varphi}(r_{i_s})$ and, if the minimum $I$-value of $r_i$ is $m$,  the number of subwords $wx^{-1}$ of $r_i$ with $I$-value $m$ is different from the number of subwords $wx$ with $I$-value $m$.
\end{definition}

The condition on the number of subwords with rightmost letter $x$ in the previous definition will be used in the next result to prove that certain $2$-cell $R_{i,j}$ corresponding to the relator $r_i$ properly involves a $1$-cell corresponding to the generator $x$.

\begin{thm} \label{weak}
     Let $\mathcal{P} = \langle a_1, \ldots, a_{n} \ | \ r_1, \ldots, r_k \rangle$ be a group presentation with cyclically reduced relators of a group $G$ with $H_1(G)$ nontrivial free abelian of rank $n-k$. Let $\varphi : G\to \mathbb{Z}$ be a surjective homomorphism. If $\{r_1,\ldots,r_k\}$ is weakly concatenable with respect to $\varphi$, then $G$ is locally indicable. 
 \end{thm}

 \begin{proof}
     We consider as before the covering $\tk$ of $\k_{\mathcal{P}}$ corresponding to $\varphi$.  Again, the $1$-cells $a_{i,j}$ are oriented from $j$ to $j+\varphi(a_i)$ (recall that we assume that the weights $\varphi(a_i)$ are non-negative) and $R_{i,j}$ are the $2$-cells  covering $R_i$  with minimum $0$-cell $j$. By Theorem \ref{red} it suffices to prove that every finite connected $2$-subcomplex $K\subset \tk$ is reducible. Note again that  the no-proper-power condition follows from \cite[Proposition 1]{ad1} and the fact that $\tk$ is an infinite cyclic cover of a $2$-complex $\kp$ with the homology of a wedge of circles. 

     Given a finite and connected $2$-subcomplex $K$, let $j\in\Z$ be the minimum $0$-cell in $K$ that is incident to some $2$-cell of $K$. Let
     $R_{i_1,j},\ldots,R_{i_s,j}$ with $i_1 < \ldots < i_s$ be the $2$-cells incident to the $0$-cell $j$. By hypothesis, $a_{i_s,j}$ is a face of the $2-$cell $R_{i_s,j}$ and is not face of any other $2$-cell $R_{i,l}$ of $K$. Note that this $1$-cell is properly involved in $R_{i_s,j}$ by definition of weakly concatenable, and therefore it is an elementary reduction. 
    \end{proof}

\begin{ej}
	Consider the presentation $\mathcal{P}=\langle a, b, c \ | \ r_1, r_2\rangle$ with relators
$$r_1 = a^{-1}c^{-1}aaab^{-1}b^{-1}c^{-1}ab, \ \ r_2 = c^{-1}b^{-1}cb^{-1}caba^{-1}ba^{-1}.$$
 Let $\varphi:G\to \Z$ be the homomorphism defined by
$\varphi(a)=\varphi(b)=\varphi(c)=1$. It is easy to verify that the multisets of minima are $m_\varphi(r_1)=\{c,a,c,a\}$ and   $m_\varphi(r_2)=\{b,c,b,c\}$, and that $\{r_1,r_2\}$ is weakly concatenable with the ordering $m_\varphi(r_2), m_\varphi(r_1)$. By Theorem \ref{weak}, $G(\pe)$ is locally indicable.
\end{ej}

\section{Applications to labelled oriented trees}

We now apply the results of the previous section to derive new results on LOT groups. The first one is a straightforward application of Theorem \ref{main} and is one of the main results of this paper. It is an extension of a well known result of Howie \cite[Theorem 10.1]{h2} on the left and right graphs associated to a labelled oriented tree, and is also related to similar results by Gersten \cite[Proposition 4.15]{ge} and Wise  \cite{wi} on Adian presentations.

Recall that a labelled oriented graph (LOG) $\Gamma$ consists of two sets $V(\Gamma)$ and $E(\Gamma)$ of vertices and edges, and three  functions $i, t, \lambda : E \to V$ that map each edge to its initial vertex, terminal vertex and label respectively. It is called a labelled oriented tree (LOT) when the underlying graph is a tree. The standard presentation $P(\Gamma)$ (of a group $G(\Gamma)$) associated to each LOG $\Gamma$ is the following: $$P(\Gamma) = \langle V(\Gamma) \ | \ \{t(e)^{-1}\lambda(e)^{-1}i(e)\lambda(e) : e \in E(\Gamma)\}\rangle$$ 
Note that all relators have total exponent $0$ and, if $\Gamma$ is a LOT, $\pg$ has deficiency $1$ and it is a homology circle. Every LOT can be changed, by a finite sequence of transformations, to a reduced one with the same homotopy type (see \cite[Section 3]{h2}). In particular one can assume that $\lambda(e)\neq i(e)$ and $\lambda(e)\neq t(e)$ for every edge $e$, which implies that all the relators of $\pg$ are cyclically reduced.

In this section we will use the homomorphism $\varphi:G(\Gamma)\to\Z$ that sends every generator of $\pg$ to $1$. The sequence of $I$-values with respect to $\varphi$ will be called $I$-sequence. Note that, given any relator $r$ (corresponding to an edge $e$ of $\Gamma$), its $I$-sequence attains exactly two minima, one at $\lambda(e)$ and the other one at $i(e)$, then $m(r)=\{\lambda(e), i(e)\}$. Analogously, it attains two maxima, one  at $\lambda(e)$ and the other one at $t(e)$.

In \cite{h2} Howie introduced two graphs $I(\Gamma)$ and $T(\Gamma)$ associated to any LOG $\Gamma$. The vertex set of the left graph $T(\Gamma)$ is $V(\Gamma)$ and, for every edge $e$ in $\Gamma$, we put an (unoriented) edge $\{\lambda(e), i(e)\}$ (note that these are precisely the letters in which the minima of the $I$-sequence are attained). Similarly, the right graph $I(\Gamma)$ has an edge  $\{\lambda(e), t(e)\}$ for every edge $e$ in $\Gamma$.  In \cite{ge} Gersten investigated these graphs in the more general context of Adian presentations. An Adian presentation is one of the form $\pe=\langle A \ | \ u_i=v_i \ (i\in J)\rangle$ where $u_i$ and $v_i$ are nontrivial positive words on $A$. In this context, the edges of the graph  $T(\pe)$ (resp. $I(\pe)$) connect the first (resp. last) letter of $u_i$ to that of $v_i$. Howie proved that, if $I(\Gamma)$ or $T(\Gamma)$ is a tree then $G(\Gamma)$ is locally indicable (and, in particular, $\Gamma$ is aspherical) \cite[Theorem 10.1]{h2}. In the case of Adian presentations, Gersten proved that if the presentation is cycle-free (i.e. if both graphs are forests) then $\kp$ is DR (diagrammatically reducible) and, in particular, aspherical \cite[Proposition 4.12]{ge}. Moreover, when $\ell(u_i)=\ell(v_i)$ for every $i$ (as in the case of LOTs), if either  $T(\pe)$ or $I(\pe)$ has no cycles then $\kp$ is DR (and hence aspherical) \cite[Proposition 4.15]{ge}. More recently, Wise showed that the $2$-complexes associated to cycle-free Adian presentations have non-positive sectional curvature, which implies that the corresponding groups are locally indicable \cite[Theorem 11.4]{wi}.

Note that Theorem \ref{main} imposes conditions to all the relators of the presentation but one. If either $I(\Gamma)$ or $T(\Gamma)$ has only one cycle, we can remove an edge of the cycle to get a forest. Then, as an immediate application of Theorem \ref{main} (or the analogous result for maxima), we obtain the following generalization of \cite[Theorem 10.1]{h2}.

\begin{corollary} \label{ciclos}
    Let $\Gamma$ be a LOT. If either $I(\Gamma)$ or $T(\Gamma)$ has at most one cycle, then $G(\Gamma)$ is locally indicable.
\end{corollary}

\begin{ej} \label{ejemplo}
Consider the following LOT $\Gamma$.

\begin{center}
\begin{tikzpicture}[scale = 0.4pt]
\SetGraphUnit{4}

\SetUpEdge[lw = 1pt]

\begin{scope}[VertexStyle/.style = {draw=none}]
\Vertex{0}
\EA(0){1}
\EA(1){2}
\SO(0){d}
\WE[unit = 2](d){a}
\SO(1){b}
\EA[unit = 2](b){e}
\SO(2){c}
\EA[unit = 2](c){f}

\draw[->, thick] (1)--(0) node[pos=.5, above = 0.3mm]  {a};
\draw[->, thick] (1)--(2) node[pos=.5, above = 0.3mm]  {d};

\draw[->, thick] (0)--(a) node[pos=.45, left=-0.3mm]  {b};
\draw[->, thick] (d)--(0) node[pos=.5, right=-0.7mm]  {e};
\draw[->, thick] (1)--(b) node[pos=.5, left=-0.7mm]  {c};
\draw[->, thick] (1)--(e) node[pos=.45, right=-0.2mm]  {f};
\draw[->, thick] (c)--(2) node[pos=.5, left=-0.7mm]  {a};
\draw[->, thick] (2)--(f) node[pos=.5, right=-0.2mm]  {a};
\end{scope}
\end{tikzpicture}
\end{center}

Neither $I(\Gamma)$ nor $T(\Gamma)$ is a tree, so Howie's (or Gersten's) result do not apply. Note that $T(\Gamma)$ has exactly one cycle and then, by Corollary \ref{ciclos}, the group $G(\Gamma)$ is locally indicable (and $\Gamma$ is aspherical).
\end{ej}

In \cite{h2} Howie studied some subfamilies of LOTs of diameter $4$ with $3$ non-extremal vertices. Using Corollary \ref{ciclos} one can show local indicability for most examples of the remaining subfamilies. Example \ref{ejemplo} belongs to one of them, in which one of the extremal vertices appears three times as a label. 

Corollay \ref{ciclos} can be generalized to a more general class of Adian presentations. Suppose $\pe=\langle a_1,\ldots,a_n\ |\ u_1=v_1,\ldots,u_{n-1}=v_{n-1}\rangle$ is an Adian presentation of deficiency $1$ with $H_1(G(\pe))=\Z$ (where the relators are cyclically reduced). Suppose further that there exists a surjective homomorphism $\varphi:G(\pe)\to\Z$ with $\varphi(a_i)>0$ for every $a_i$. Note that for any such homomorphism $\varphi$, the multiset of minima $m_\varphi(r_i)$ of each relator $r_i$ is the set consisting of the first letter of $u_i$ and the first letter of $v_i$. Then, if either $I(\pe)$ or $T(\pe)$ has at most one cycle, as an immediate application of Theorem \ref{pesos}, we obtain the following result, which extends, for a subfamily of Adian presentations, Wise's result \cite[Theorem 11.4]{wi}.

\begin{corollary}\label{adian}
Let $\pe=\langle a_1,\ldots,a_n\ |\ u_1=v_1,\ldots,u_{n-1}=v_{n-1}\rangle$ be an Adian presentation of deficiency $1$ with $H_1(G(\pe))=\Z$. Suppose further that there exists a surjective homomorphism $\varphi:G(\pe)\to\Z$ with $\varphi(a_i)>0$ for every $a_i$. If either $I(\pe)$ or $T(\pe)$ has at most one cycle, then $G(\pe)$ is locally indicable.
\end{corollary}

Note that the condition on the existence of the homomorphism $\varphi$ in the previous corollary is automatically satisfied when $\ell(u_i)=\ell(v_i)$ for every $i$.

We finish the paper by analyzing a strategy that can be useful to study local indicability of LOT groups for which $I(\Gamma)$ and $T(\Gamma)$ have more than one cycle. The main idea is to apply convenient extended Nielsen transformations to the LOT presentation $\pg$ to obtain an equivalent presentation $\pe$ which satisfies the hypotheses of Theorem \ref{weak}.

For now on we consider the labelled oriented (original) versions of $I(\Gamma)$ and $T(\Gamma)$: for every edge $a \xrightarrow{c}b$ of a LOT $\Gamma$, we put an edge $a \xrightarrow{b} c$ in $T(\Gamma)$ and an edge $c \xrightarrow{a} b$ in $I(\Gamma)$. We state the result for $T(\Gamma)$, the corresponding result for $I(\Gamma)$ is analogous.

\begin{definition}
    Let $\Gamma$ be a LOT. Let $C$ be a cycle in $T(\Gamma)$ and $e$ an edge in $C$. We say that $e$  is \textit{properly labelled} in $C$ if the number of edges in $C$ with the same label as $e$ oriented clockwise is different from the number of edges with the same label as $e$ oriented counterclockwise.
  \end{definition}

\begin{lemma} \label{ciclo}
    Let $\Gamma$ be a LOT. Let $e$ be an edge in a simple cycle $C$ of $T(\Gamma)$ with label $x$. Suppose further that $x$ does not appear as a vertex in $C$. Then, by applying extended Nielsen transformations, one can replace the relator $r$ corresponding to $e$ in $P(\Gamma)$  by another  relator $\Tilde{r}$ which contains the generator $x$ and other generators that appear as labels or vertices in $C$, the new relator $\Tilde{r}$ has a constant $I$-sequence (it attains a minimum at every letter), and in particular it attains a minimum at $x$. Moreover, if $e$ is properly labelled in $C$, the number of initial subwords of $\Tilde{r}$ of the form $wx$ and the number of initial subwords  of the form $wx^{-1}$  are different.
\end{lemma}

\begin{proof}
    We say that a word is a \textit{zig-zag} if its length is even and the exponent of the $i$-th letter is $(-1)^{i+1}$ (i.e. it has the form $x_1x_2^{-1}\ldots x_{2s}^{-1}$). Suppose $r_1$ and $r_2$ are relators and $a$ is a letter in which both relators attain a minimum in the $I$-sequence. Let $b$ and $c$ be the other letters in which $r_1$ and $r_2$ attain the other minimum. Up to cyclic permutation and inversions, we can assume $r_1 = bw_1a^{-1}$ and $r_2 = aw_2c^{-1}$ with $w_1,w_2$ zig-zags. Now, we can change $r_1$ by $r_1r_2 = bw_1w_2c^{-1}$ which attains exactly two minima in its $I$-sequence, one at $b$ and the other one at $c$. 

    Let $C = a_1,a_2,a_3,\ldots,a_l,a_1$ be the cycle of $T(\Gamma)$ and $e = (a_1,a_2)$. Let $r_i$ be the relator corresponding to $(a_i,a_{i+1})$ for every $1 \leq i \leq l$ (where the indices are taken modulo $l$). Up to cyclic permutations and inversions, we can assume that $r_i = a_iw_ia_{i+1}^{-1}$ where $w_i$ is a zig-zag word. Now we can replace $r_1$ by $r_1r_2\cdots r_l = a_1w_1w_2\cdots w_la_1^{-1}$, which, after a cyclic permutation, is $w_1w_2\cdots w_l$. Note that $r=w_1w_2\cdots w_l$ is a zig-zag. Then every letter attains the same $I$-value (the minimum). Finally, note that since $(a_1,a_2)$ is properly labelled, the number of occurrences of $x$ in this new relator is different from the number of occurrences of $x^{-1}$.    
\end{proof}

\begin{thm} \label{ciclosetiq}
    Let $\Gamma$ be a LOT. Denote by $X_0, X_1, \ldots, X_k$ the connected components of $T(\Gamma)$. Suppose that $C_1, \ldots, C_{k}$ are generating simple cycles in $X_0$ and the other connected components $X_1,\ldots, X_k$ are trees. If for every $1 \leq i \leq k$ there is at least one properly labelled edge of $C_i$ that has a vertex of $X_i$ as a label and,  for $j \geq i$  there is no other vertex of $X_j$  that appears as a label in the edges of $C_i$, then $G_{\Gamma}$ is locally indicable.   
\end{thm}
	
\begin{proof}
 To illustrate the idea of the proof, we show the case $k=1$.
    Let $e_1$ be a properly labelled edge in the cycle $C_1$, which is labelled with a vertex $x_1$ of $X_1$. By 
    Lemma \ref{ciclo} we can replace the corresponding relator $r_1$ by a relator $\Tilde{r_1}$, which attains a minimum at the vertex $x_1$ and in which the number of initial subwords $wx_1$ with minimum $I$-value and the number of initial subwords $wx_{1}^{-1}$ with minimum $I$-value are different. Note that the new presentation $\pe$ is Andrews-Curtis equivalent to $P(\Gamma)$.

    Now, the relators of the new presentation $\pe$ are weakly concatenable and the result follows from Theorem \ref{weak}. The ordering  for a weak concatenation from back to front is the following: first follow the order of collapses of all the edges of the tree $X_1$ to the vertex $x_1$, then continue with the new relator $\Tilde{r_1}$, and finally use any ordering of weak concatenation for the remaining relators of $X_0$ (note that $X_0-e_1$ is a tree).
    
     The general case follows by induction in a similar way. Note that we require that for $j \geq i$  there is no other vertex of $X_j$  that appears as a label in the edges of $C_i$. This condition is necessary to guarantee weak concatenability, since, by Lemma \ref{ciclo} the new relators  $\Tilde{r_i}$ attain a minimum at every letter that appears in the cycle $C_i$.
\end{proof}

\begin{ej}
    Consider the following LOT $\Gamma$

   \begin{center}
   	\begin{tikzpicture}[scale = 0.4pt]
   		\SetGraphUnit{4}
   		
   		\SetUpEdge[lw = 1pt]
   		
   		\begin{scope}[VertexStyle/.style = {draw=none}]
   			\Vertex{0}
   			\EA(0){1}
   			\EA(1){2}
   			\SO(0){d}
   			\WE[unit = 2](d){e}
   			\SO(1){c}
   			\WE[unit = 2](c){f}
   			\EA[unit = 2](c){b}
   			\SO(2){a}
   			
   			\draw[->, thick] (0)--(1) node[pos=.5, above = 0.3mm]  {d};
   			\draw[->, thick] (2)--(1) node[pos=.5, above = 0.3mm]  {e};
   			
   			\draw[->, thick] (0)--(e) node[pos=.45, left=-0.3mm]  {a};
   			\draw[->, thick] (0)--(d) node[pos=.5, right=-0.7mm]  {b};
   			\draw[->, thick] (c)--(1) node[pos=.5, left=-0.2mm]  {f};
   			\draw[->, thick] (f)--(1) node[pos=.5, left=-0.2mm]  {b};
   			\draw[->, thick] (b)--(1) node[pos=.45, right=-0.2mm]  {a};
   			\draw[->, thick] (a)--(2) node[pos=.5, left=-0.7mm]  {c};
   		\end{scope}
   	\end{tikzpicture}
   \end{center}

Both $I(\Gamma)$ and $T(\Gamma)$ have two cycles, so we cannot apply Corollary \ref{ciclos}. The graph $I(\Gamma)$ is displayed in
Figure \ref{igamma}. 

\begin{figure}[h]

\begin{center}
\includegraphics[width=4cm]{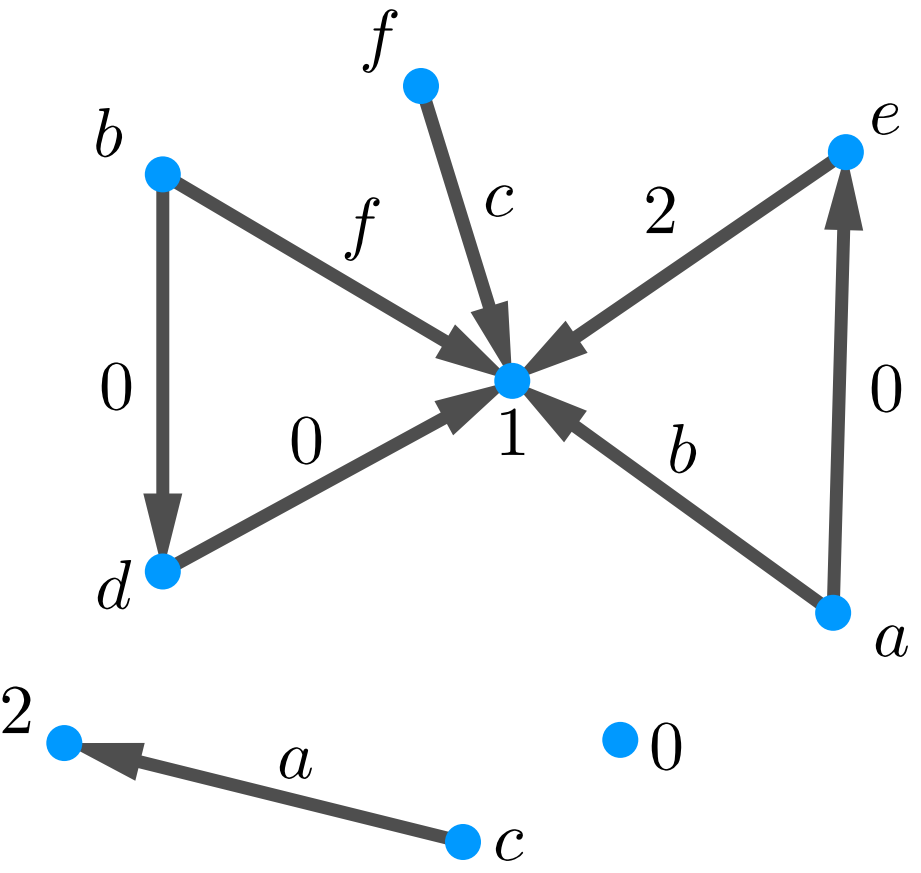}
\end{center}
	\caption{$I(\Gamma)$\label{igamma}}
\end{figure}

Note that $I(\Gamma)$ satisfies the hypotheses of Theorem \ref{ciclosetiq} with $X_0$ the subgraph with vertices  $\{b,d,1,e,a,f\}$, $X_1 = \{0\}$ and $X_2$ the tree with vertices $\{2,c\}$. In the cycle $C_1 = 1,b,d,1$ we consider the edges properly labelled with $x_1=0$ and in  the cycle $C_2 = 1, a, e, 1$ we take the edge properly labelled with $x_2=2$. This shows that $G(\Gamma)$ is locally indicable.
\end{ej}

\end{document}